\documentclass[11pt]{article}
\usepackage [dvips]{graphics}
\usepackage[centertags]{amsmath}
\usepackage{amsfonts}
\usepackage{amssymb}
\usepackage{amsthm}
\usepackage{newlfont}
\usepackage{stmaryrd}
\usepackage{color,epsfig}
\usepackage{amsfonts,graphicx,psfrag}
\usepackage{graphicx}
\usepackage{float}

\usepackage{txfonts}
\usepackage{mathrsfs}

\theoremstyle{plain}
\newtheorem{thm}{Theorem}[section]
\newtheorem{cor}[thm]{Corollary}
\newtheorem{lem}[thm]{Lemma}

\newtheorem{rem}[thm]{Remark}

\def\sqr#1#2{{\vcenter{\vbox{\hrule height.#2pt
              \hbox{\vrule width.#2pt height#1pt \kern#1pt \vrule
width.#2pt}
              \hrule height.#2pt}}}}
\def\be{\begin{equation}}
\def\ee{\end{equation}}

\def\ga{{\gamma}}

\def\ep{{\epsilon}}

\def\lb{\label}

\def\ga{{\gamma}}

\def\d{{d\over dt}}
\def\sd{{d^2\over dt^2}}

\def\no{\noindent}

\def\bs{\bigskip}

\def\dim{\hbox{\rm dim$\,$}}

\def\({\Big (}
\def\){\Big )}
\def\[{\Big[}
\def\]{\Big]}

\def\be{\begin{equation}}
\def\bel{\begin{equation}\label}
\def\ee{\end{equation}}
\def\bea{\begin{eqnarray}}
\def\eea{\end{eqnarray}}
\def\bt{\begin{theorem}}
\def\et{\end{theorem}}
\def\bc{\begin{corollary}}
\def\ec{\end{corollary}}
\def\bl{\begin{lemma}}
\def\el{\end{lemma}}
\def\bp{\begin{proposition}}
\def\ep{\end{proposition}}
\def\br{\begin{remark}}
\def\er{\end{remark}}
\def\ba{\begin{array}}
\def\ea{\end{array}}
\def\bd{\begin{definition}}
\def\ed{\end{definition}}

\makeatletter
   
   \@addtoreset{equation}{section}
\makeatother

\begin{document}

\title{\bf Eigenvalue problem of Sturm-Liouville systems with separated boundary conditions }
\author{Xijun Hu\thanks{Partially supported
by NSFC(No.11425105, 11131004), NCET,  E-mail:xjhu@sdu.edu.cn }
 \quad Penghui Wang\thanks{ Partially supported by NSFC(No.11471189),
 E-mail: phwang@sdu.edu.cn }    \\ \\
 Department of Mathematics, Shandong University\\
Jinan, Shandong 250100, The People's Republic of China\\
}
\date{}
\maketitle
\begin{abstract}

Let $\lambda_j$ be  the $j$-th eigenvalue  of Sturm-Liouville systems with
separated boundary conditions, we build up the Hill-type formula, which represent  $\prod\limits_{j}(1-\lambda_j^{-1})$ as a determinant of finite matrix. This is  the first attack on such a formula under non-periodic type   boundary conditions. Consequently, we get the Krein-type
trace formula based on the Hill-type formula, which  express $\sum\limits_{j}{1\over
\lambda_j^m}$  as trace of finite  matrices.   The trace formula can be used to  estimate the conjugate point alone a geodesic in Riemannian  manifold and to get some infinite sum identities.
\end{abstract}

\bs

\no{\bf AMS Subject Classification:} 34B24, 34L15,  47E05

\bs

\no{\bf Key Words}. Hill-type formula, trace formula, Hamiltonian
systems, Sturm-Liouville systems

\section{Introduction}

In this paper, we will consider the eigenvalue problem for the Sturm-Liouville systems
 \begin{eqnarray}
-(P\dot{y}+Qy)^\cdot+Q^T\dot{y}+(R+\lambda R_1)y=0, \label{sl1} \end{eqnarray}
where  $Q$ is a continuous path of $n\times n$ matrices, and $P, R, R_1$ are continuous paths of $n\times n$ symmetric matrices on $[0,T]$. Instead of Legendre convexity condition, we assume that for any $t\in [0,T]$,  $P(t)$  is  invertible.

  The
eigenvalue problem of the Sturm-Liouville systems depends on the
boundary conditions. There are two important type boundary
conditions, periodic type and separated type. For the literature in $n$-body problem, readers
can  refer to \cite{FT},\cite{HS}. The eigenvalue problem for $S$-periodic  boundary value problem, that is,
$y(0)=Sy(T)$ for some orthogonal matrix $S$, was studied in
\cite{HOW}.
 In
the present paper, we will consider the  separated boundary
conditions, which includes the homogenous  Dirichlet, Neumann and Robin
boundary conditions. More precisely,
 let $\Lambda_0, \Lambda_1$ be two Lagrangian subspaces of $(\mathbb R^{2n}, \omega_0) $  which are  phase spaces with standard symplectic structure.  Set  $x=P\dot{y}+Qy$, $z=(x,y)^T$, and the separated boundary condition is given by
\bea z(0)\in\Lambda_0,\quad  z(T)\in\Lambda_1. \label{slb1} \eea

In order to understand the eigenvalue problem of the system
(\ref{sl1}-\ref{slb1}),  for the eigenvalues $\lambda_j$ we will build up a  formula with the form
$\prod_{j}(1-\lambda_j^{-1})=\det(\mathcal{M}) $, where the matrix
$\mathcal{M}$ depends mainly on the monodromy matrix and the Lagrangian subspaces $\Lambda_0,\Lambda_1$. We called it Hill-type formula because a similar formula for periodic orbits was shown by Hill when he considered the motion of lunar perigee \cite{Hi} at 1877. However,
Hill did not prove the convergence of the infinite determinant, and
the convergence was given by Poincar\'e \cite{Po}. Afterwards, the
Hill-type formula for a periodic solution of Lagrangian system on
manifold was given  by Bolotin \cite{B}. For more results, please
refer (\cite{BT},\cite{De},\cite{HW1},\cite{Dav}) etc. We should point out that, till now,
all the known results on the Hill-type formula were given for the
periodic-type boundary problem.

To state Hill-type formula for the separated boundary conditions, we
firstly introduce some notations.   Suppose $\Lambda$ is a
Lagrangian subspace of $(\mathbb R^{2n}, \omega_0) $, a Lagrangian
frame for $\Lambda$ is a linear map $Z: \mathbb R^n \rightarrow
\mathbb R^{2n}$ whose image is $\Lambda$. It is easy to see that the
frame is of the form $Z=\left(\begin{array}{cc}X
\\ Y \end{array}\right)$, where $X,Y$ are $n\times n$ matrices and
satisfied $X^TY=Y^TX$.

By the standard Legendre transformation,  the linear  system
(\ref{sl1}) with the boundary conditions (\ref{slb1}) corresponds to the linear Hamiltonian system,
 \bea \dot{z}=JB_\lambda(t)z, \quad z(0)\in\Lambda_0,\quad  z(T)\in\Lambda_1, \label{h1}\eea
with \bea  B_\lambda(t)=\left(\begin{array}{cc}P^{-1}(t)& -P^{-1}Q(t) \\
-Q(t)^TP^{-1}(t)  & Q(t)^TP^{-1}(t)Q(t)-R(t)-\lambda R_1(t)
\end{array}\right).\label{b2} \eea
 Without confusion, for Lagrangian system, denote  $\gamma_\lambda(t)$  the
fundamental solution of (\ref{h1}), that is
$\dot{\gamma}_\lambda(t)=JB_\lambda(t)\gamma_\lambda(t)$ with
$\gamma_\lambda(0)=I_{2n}$.
Let $Z_0,Z_1$ be frames
of $\Lambda_0,\Lambda_1$. Obviously, $\gamma_\lambda(T) Z_0$ are
frames of $\ga_\lambda(T) \Lambda_0 $. $(\ga_\lambda(T) Z_0,Z_1)$
are   $2n\times 2n$ matrices.

To simplify the notation, set $A=-\d(P\d+Q)+Q^T\d+R$, which is a
self-adjoint operator on $L^2([0,T], \mathbb R^{n})$ with domain:
$$   D(\Lambda_0,\Lambda_1)=\{y\in W^{2,2}([0,T],  \mathbb R^{n}),   z(0)\in\Lambda_0,   z(T)\in\Lambda_1  \}.      $$
Throughout of the paper, without loss of generality,  we will assume
$A$ is  nondegenerate,  that is, $0$ is not an eigenvalue of
(\ref{sl1}-\ref{slb1}). It is obvious that $\lambda$ is a nonzero
eigenvalue of the system(\ref{sl1}-\ref{slb1}) if and only if
$-{1\over \lambda}$ is an eigenvalue of $R_1A^{-1}$. In what follows, the multiplicity of an eigenvalue $\lambda_j$ means the algebraic multiplicity of $R_1A^{-1}$ at ${-1/\lambda_j}$.

\begin{thm}\label{thm1.1}
Under the nondegenerate  assumption, we have
\begin{eqnarray}\label{t.1a}
\prod_{j}(1-\lambda_j^{-1})=\det (\ga_{1}(T) Z_0,Z_1) \cdot
\det(\ga_0(T) Z_0,Z_1)^{-1},     \end{eqnarray} where the left
infinite product takes on the eigenvalues $\lambda_j$  counting the multiplicity.
\end{thm}

\begin{rem}
The Hill-type formula for $S$-periodic orbits was built up in
\cite{HOW} with the following form \bea
\prod_{j}(1-\lambda_j^{-1})=\det (\ga_{1}(T)-S_d) \cdot
\det(\ga_0(T)-S_d),\lb{s.1} \eea where $S_d=diag(S,S)$. Although
 (\ref{t.1a}) is similar to (\ref{s.1}), the proof is different.
(\ref{s.1}) is derived from the Hill-type formula for $S$-periodic
orbits of Hamiltonian systems \cite{HW}. The proof of  (\ref{t.1a})
is  direct, and could cover the case of (\ref{s.1}). To the best of our knowledge,
(\ref{t.1a})  is the first study on the Hill-type formula of
non-periodic type boundary problem. The corresponding formula
 in Hamiltonian systems for the orbits
with Lagrangian boundary conditions is still open.
\end{rem}

 Trace formula is a powerful tool in the study of eigenvalue problem, especially in estimating the first eigenvalue.
  The first work on the trace formula was established by Krein\cite{K1,K2} for the $-1$-periodic orbits in the simple case.
   For the  system with $S$-periodic boundary condition, the trace formula was established in \cite{HOW,HW1}.
The present paper is  a continuous work of \cite{HOW,HW1}, and we will build up the trace formula for separated
boundary value problem of Sturm-Liouville system. The idea to get the trace formula is similar to
that in \cite{HOW}. That is, using $\lambda R_1$ instead of
$R_1$, and give Taylor expansion on both sides of the Hill-type
formula. With the notations defined in Section \ref{sec3}, we have
the following theorem.

\begin{thm}\label{thm1.2} Assume $A$ is non-degenerate, $\lambda_j$ are eigenvalues of the Sturm-Liouville system (\ref{sl1}-\ref{slb1}) counting the multiplicity,
 we have
for any positive integer $m$,
\begin{eqnarray}\label{0.0.0}
\sum\limits_{j}{1\over \lambda_j^m}=m\sum_{k=1}^m
\frac{(-1)^{k}}{k}\[\sum\limits_{j_1+\cdots+j_k=m}Tr(G_{j_1}\cdots
G_{j_k})\],
\end{eqnarray}
where $G_k, k\in\mathbb{N}$, defined in (\ref{G}), are $n\times n$
matrices.
\end{thm}

For applications, a main observation is that the trace formula can be used to estimate the non-degeneracy of the system.
  Moreover, we can estimate the relative Morse index and Maslov-type
   index. The Maslov-type index is a powerful tool in study the
   stability problem, please refer \cite{Lon4},\cite{Lon2},\cite{HS}
   for the detail. By using the trace formula and Maslov-type index theory, in \cite{HOW}
    we studied the stability region and hyperbolic region of elliptic Lagrangian orbits in planar three body problem.

 It is not hard to see that all the results on the applications of trace formula in \cite{HOW} have twins in the case of separated  boundary conditions.
  We will not list all the
theorems here, but give some computations for the trace formula some special  case.
Let $R$ be a  continuous path of $n\times n$ symmetric matrices on $[0,T]$, we
consider the system \bea \ddot{y}+\lambda Ry =0 \lb{4.0}. \eea Set
$ R^+=\frac{1}{2}(R+|R|)$,  which is a path of nonnegative symmetric
matrices, we have the following corollary.  \begin{cor}\lb{cr1.3}   Suppose $Tr\Big(\int_0^T\(t-\frac{t^2}{T}\)R^+dt\Big)<1$, the Dirichlet  problem for (\ref{4.0}) has no nontrivial solution.

\end{cor}
This result can  be used  to estimate the conjugate point alone
a geodesics in Riemannian  manifold. For reader's convenience, we give details here.
Let  $c:[0,a]\to\mathcal {M} $ be a  geodesic of  Riemannian manifold $\mathcal {M}$.
Choose $\{e_1(0),...,e_n(0)\}$ to be an orthogonal normal basis of
$\dot{c}(0)^\perp\subset T_{c(0)}\mathcal{M}$. Its parallel
transport $\{e_1(t),...,e_n(t)\}(t\in[0,a])$ along $c$ gives an
orthogonal normal basis of $\dot{c}(t)^\perp\subset
T_{c(t)}\mathcal{M}$.  Recall that
the Jacobi equation is
$${D^2J\over dt}+R(\dot{c},J(t))\dot{c}=0. $$
The point $c(t_0)$ is said to be conjugate to $c(0)$ along $c$, $t_0\in[0,a]$, if there exists a  nontrivial Jacobi field $J$ along $c$, with $J(0)=J(t_0)=0$.
  Suppose
$J(t)=\sum_{i=1}^nJ_i(t)e_i(t)$ is the Jacobi field along $c$, then
Jacobi equation can be rewritten as  \bea
\ddot{J}_i(t)+\sum_{j=1}^nR_{ij}(t)J_j(t)=0,\,\,
i=1,...,n,\lb{2.1}\eea where $R_{ij}(t)=\langle
R(\dot{c}(t),e_j(t))\dot{c}(t),e_i(t)\rangle$. Let
$R(t)=(R_{ij}(t))$, which is a symmetric matrix, then $c(t_0)$ is conjugate point if and only if   the
second order system $\ddot{X}(t)+R(t)X(t) $ with Dirichlet boundary conditions has a nontrivial solution on $t\in[0,t_0]$.  It is obvious that
$ \hat{R}(t):=Tr(R(t))  $  is the Ricci curvature  in the direction of  $\dot{c}$. Set $\hat{R}^+(t)=\frac{1}{2}(\hat{R}+|\hat{R}|)$, then Corollary \ref{cr1.3}
implies that there is no conjugate point alone $[0,a]$ if $\Big(\int_0^T\(t-\frac{t^2}{T}\)\hat{R}^+dt\Big)<1 $.

Now, if we consider the system (\ref{4.0}) in the case $n=1, R=1$ with the boundary conditions  \bea y(0)=0,\quad \cos(\theta) y(T)+\sin( \theta) \dot{y}(T)=0, \,\ \theta\in[0,\pi/2]. \lb{mix} \eea It is well known that the $k$-th eigenvalue $\lambda_k$ is the $k$-th positive solution of the next transcendental  equation
\bea \tan (\sqrt{\lambda}T)=-\tan(\theta)\sqrt{\lambda}.  \lb{lamk}\nonumber\eea
It is easy to check that if $\theta=0$, then $\lambda_k={k^2\pi^2\over T^2}$, and if  $\theta=\pi/2$,
 then $\lambda_k=\frac{\pi^2}{T^2}(k-\frac{1}{2})^2$.
For $\theta\in(0,\pi/2)$, it is obvious that $(k-\frac{1}{2})\pi<\sqrt{\lambda_k}T<k\pi$.   However,  $\lambda_k$ can only be solved numerically. As an application of the trace formula, we have the following equality, which itself is interesting.
\bea \sum\limits_{k\in\mathbb{N}}{1\over \lambda_k}=\frac{3T^2\sin(\theta)+T^3\cos(\theta)}{6(\sin(\theta)+T\cos(\theta))}. \lb{e.e} \eea
Obviously, for $\theta=0$, (\ref{e.e}) gives the well known identity
$ \sum\limits_{k\in\mathbb{N}}\frac{1}{k^2}=\frac{\pi^2}{6}$,  and  for $\theta=\pi/2$, (\ref{e.e})
 gives the identity
$  \sum\limits_{k\in\mathbb{N}}\frac{1}{\pi^2(k-\frac{1}{2})^2}= \frac{1}{2} $.
 To the best of our knowledge, for $\theta\in(0,\pi/2)$, we don't know any such kind of formula before on the sum of $1\over \lambda_k$.   The detailed calculation will be listed in Section 4.   Moreover, it is worth to point out that  we can compute the value of $\sum{1\over \lambda^m_k}$ for any $m\in \mathbb{N}$ by the trace formula (\ref{0.0.0}).

The present paper is organized as follows. In section 2, we give the
proof of the Hill-type formula (\ref{t.1a}). Section 3 is devoted to
proving the trace formula. Finally, in Section 4,  we will give the proof of Corollary \ref{cr1.3} and identities (\ref{e.e}).

\section{Hill-type formula for Sturm-Liouville systems }\label{sec2}

In this section, we will give the proof of the Hill-type formula.
The following lemma coming from \cite[Lemma 3.6]{Si} plays a
important role.
\begin{lem}\label{lem2.0}
Let $f(z)$ be an entire function with zeros at $z_1,z_2,\cdots$
(counting multiplicity). Suppose $f$ satisfied
\\  i)  Exponential bounded condition: for any $\epsilon$,  there exist $C_\epsilon$ such that
\bea| f(z)| \leq C_\epsilon \exp(\epsilon |z|), \label{c1}
\nonumber\eea ii) Sum finite condition: $\sum_{n=0}^\infty
|z_n|^{-1}<\infty, \label{c2}  $ then \bea f(z)=
f(0)\prod_{i=1}^{\infty} (1-z_n^{-1}z). \nonumber\label{c3}\eea
\end{lem}

 Since $R_1A^{-1}$ is a trace class operator, by \cite[Chapter 3, P33]{Si}, we have that $\det(I+\lambda R_1 A^{-1})$ is an entire function with
$  |\det(I+\lambda R_1A^{-1}) | \leq \exp(|\lambda|\cdot
\|R_1A^{-1}\|_1)$,  where $\|\cdot\|_1$ is the trace norm. Moreover,
it satisfied the exponential bounded condition. It is obvious that
$\lambda_n$ is a zero point of $\det(I+\lambda R_1 A^{-1})$ if and
only if $\lambda_n$ is an eigenvalue of the system
(\ref{sl1}-\ref{slb1}).
It follows that $\sum\limits_{n}{1\over |\lambda_n|}<\infty$, where the sum takes for $\lambda_j$ counting multiplicity.
From Lemma \ref{lem2.0}, we have that \begin{eqnarray}\det(I+\lambda
R_1 A^{-1})=\prod_{n} (1-\lambda_n^{-1} \lambda).\end{eqnarray}

To continue, it is easy to verify  that $y_0$ is a solution of
(\ref{sl1}-\ref{slb1}) with respect to the eigenvalue $\lambda_0$ if
and only if $z_0$ is a solution of $(\ref{h1})$ with respect to the
same eigenvalue. And it is equivalent to
$\gamma_{\lambda_0}(T)z_0(0)\in \Lambda_1$. We have the following
observation.

\begin{lem}\label{lemma2.1} Suppose that $A$ is nondegenerate, then
\bea  \dim\ker (R_1A^{-1}+1/\lambda_0)=\dim  \gamma_{\lambda_0}(T)\Lambda_0\cap\Lambda_1. \eea
\end{lem}

For the Lagrangian frames $Z_i$ of $\Lambda_i$, $i=0,1$, set
\bea  g(\lambda)= \det(\ga_\lambda(T) Z_0,Z_1).  \label{g} \eea
  We have
\begin{lem}\label{lemma2.2}
$g(\lambda)$ is an analytic function and satisfied the  exponential bounded condition.
 \end{lem}

\begin{proof} The analyticity of $\gamma_\lambda$ comes from Krein \cite{K2} essentially. For the Taylor expansion, readers are referred to \cite[Section 2.2]{HOW}.
 Next, we will show that $g(\lambda)$ satisfies the exponential bounded condition.  For nonzero $\lambda$, let $\mu=\lambda^{1/4}$ and $a(\mu)=diag(\mu I_n, \mu^{-1} I_n)$, we set
 $\hat{\ga}_\lambda(t)=a(\mu)^{-1}\ga_\lambda(t)$,
direct computation shows that
           $$\frac{d}{dt}(\hat{\ga}_\lambda(T))=Ja(\mu)B_\lambda(t)a(\mu)\hat{\ga}(t).$$  Moreover,  $$ a(\mu)B_\lambda(t)a(\mu)= \bar{B}_{\mu}(t)+\mu^2\hat{B}(t) $$ with
$$  \bar{B}_{\mu}=\left(\begin{array}{cc}0_n& -P^{-1}Q \\
-Q^TP^{-1}  & \mu^{-2}(Q^TP^{-1}Q-R)
\end{array}\right),  \,\  \hat{B}=\left(\begin{array}{cc}P^{-1}&  0_n\\
0_n  & -R_1
\end{array}\right).$$
Let $\bar{\ga}_\mu$ be the fundamental solution with respect to
$\bar{B}_{\mu}$,   then
$$\frac{d}{dt}(\bar{\ga}_\mu^{-1}\hat{\ga}_\lambda(t))=\mu^2J\bar{\ga}_\mu^T\hat{B}(t)\bar{\ga}_\mu\cdot\bar{\ga}_\mu^{-1}\hat{\ga}_\lambda(t)).$$
Restricting  on the region  $|\mu|\geqq1$,  it is obvious that
$\bar{\ga}_\mu$ is bounded, and thus
$\bar{\ga}_\mu^T\hat{B}(t)\bar{\ga}_\mu$ is bounded.  Hence
$\|\bar{\ga}_\mu^{-1}\hat{\ga}_\lambda(T))  \| \leqq \exp(C|\mu|^2)$ for
some constant $C$. Consequently
$$\|\gamma_\lambda(T)\|\leqq C_0|\lambda|^{1/2} \exp(C|\lambda|^{1/2}).    $$  Finally, notice that $g(\lambda)$ is the finite combination of the finite product of the branches in the matrix, we have the results.

\end{proof}

By Lemma \ref{lemma2.1}, $g(\lambda_0)=0$ if and only if $\lambda_0$
is eigenvalue of (\ref{sl1}-\ref{slb1}).  That is, $g(\lambda)$ has
the same zero points as $\det(I+\lambda R_1A^{-1})$. Moreover, we
have the following lemma.
\begin{lem} \label{lemma2.3}
Suppose  $R_1>0$ and $\lambda_0$ is a zero point of $g(\lambda)$,   then the multiplicity of $g(\lambda)$ at $\lambda_0$ is same as the multiplicity of  $\det (I+\lambda R_1A^{-1})$  at  $\lambda_0$.
 \end{lem}

\begin{proof}
Suppose the multiplicity of $g(\lambda)$ and  $\det (I+\lambda R_1A^{-1})$  at $\lambda_0$ is $m_1$ and $m_2$ respectively.   Since   $R_1>0$, the eigenvalue of $R_1A^{-1}$ is simple and  then $m_2=dim\ker (I+\lambda R_1A^{-1}) $. By Lemma  \ref{lemma2.1}, we have $m_2\leqq m_1$.
On the other hand, by the techniques of small perturbation (details could be found in \cite[Section 4]{HW}),  we can assume $g(\lambda)$ has $m_1$ simple zeros near $\lambda_0$, and thus $m_1\leqq m_2$, which implies the result.
\end{proof}

From the above lemmas,  we will give the proof of Theorem 1.1.
\vskip2mm\noindent\emph{Proof of Theorem \ref{thm1.1}.} We firstly
prove the Hill-type formula for  the case $R_1>0$, by the nondegenerate  assumption,  $0$
is not a zero point of $g(\lambda)$.   Please note that  both  $\det
(I+\lambda R_1A^{-1})$  and $g(\lambda)$ satisfy the exponential
bounded conditions and by Lemma \ref{lemma2.3}, they have the same
zero points with same multiplicities. Next by Lemma \ref{lem2.0}, we
have \bea  \det (I+\lambda R_1A^{-1})= g(0)^{-1}g(\lambda).
\label{hillp}  \eea In the general case,  choose $\alpha_0 \in
\mathbb R $ such that $R_1-\alpha_0I_n>0 $ and $A+\alpha_0 I_n$ is
nondegenerate,  then \bea   \det (I+ \lambda R_1A^{-1})=  \det [I+
\lambda(R_1-\alpha_0 I_n)(A+\alpha_0I_n)^{-1}]\cdot  \det (I+
\alpha_0 A^{-1})   \label{hillp1}\eea By using (\ref{hillp}) on the
two factors of the right hand side of (\ref{hillp1}), we have
(\ref{hillp}) for the general $R_1$.  By taking $\lambda=1$ we get the desired
result (\ref{t.1a}).
 \hfill$\Box$

\section{Trace  formula for Sturm-Liouville systems }\label{sec3}

In this section, we will prove Theorem \ref{thm1.2}. To do this, we will consider the expansion
of the Hill-type formula (\ref{hillp}). Notice that $A^{-1}$ is a
trace class operator, by \cite[P47, (5.12)]{Si},
\begin{eqnarray} \det(I+\lambda R_1A^{-1})=\exp\Big(\sum\limits_{m=1}^\infty\frac{(-1)^{m+1}}{m}\lambda^m Tr((R_1A^{-1})^m)\Big).  \label{cc4.20}  \end{eqnarray}
Next, we will give the expansion on $g(\lambda)$. Let
$V_0=\Lambda_0\cap\Lambda_1 $, assume  $\dim V_0=k_0$, then $0\leq
k_0\leq n$. Suppose that $\{d_1,\cdots,d_{k_0}\}$ is an orthonormal
basis of $V_0$, and $\{d_1,\cdots,d_{k_0},d_{k_0+1},\cdots, d_n\}$
is an orthonormal basis of $\Lambda_0$. Notice that $\mathbb
R^{2n}=\Lambda_0\oplus J\Lambda_0$. Therefore, setting $d_{n+j}=J
d_j$, we have $\{d_1,\cdots, d_{2n}\}$ is a basis of $\mathbb
R^{2n}$ and the matrix $M_1=\left(d_1,d_2,\cdots,\,d_{2n}\right)$ is
a symplectic orthogonal matrix. Next, set $V_1=\Lambda_1\ominus V_0$, then it is a Lagrangian subspace of $\mathbb R^{2n}\ominus(V_0\oplus JV_0)$. Take an
orthonormal basis $\{f_{k_0+1},\cdots, f_n\}$ of $V_1$, then
$\{d_1,\cdots, d_{k_0},f_{k_0+1},\cdots, f_n\}$ is an orthonormal
basis of $\Lambda_1$.

Let $\{e_k;k=1,\cdots, 2n\}$ be the standard basis of $(\mathbb R^{2n}, \omega_0)$. Obviously $e_{n+k}=Je_k$ and $M_1^Td_j=e_j$ for  $1\leqq j\leqq n$.
Notice that  $M_1^T(f_{k_0+1},\cdots,f_{n})$ gives a Lagrangian frame of $M_1^TV_1$.
By  direct computation, for
 $k_0+1\leqq l\leqq n$ and $1\leqq j\leqq k_0$, $$(M_1^Tf_l,e_j)=(M_1^Tf_l,e_{n+j}) =0.$$   Rewrite such a frame as  $\left(\begin{array}{cc}\tilde{X}_1
\\ \tilde{Y}_1 \end{array}\right)$, where $\tilde{X}_1$, $\tilde{Y}_1$  are  $(n-k_0)\times(n-k_0)$  matrices and $\tilde{Y}_1$ is nonsingular.
Let $M_2=\left(\begin{array}{cc} I_{n-k_0}&  -\tilde{X}_1\tilde{Y}_1^{-1}\\
0_{n-k_0}  & I_{n-k_0}
\end{array}\right)$, and \begin{eqnarray}M_3=(I_{2k_0}\diamond M_2)\cdot M_1,\label{eq3.3}\end{eqnarray} where $I_{2k_0}\diamond M_2=\left(\begin{array}{cccc} I_{k_0}& 0& 0 & 0\\ 0& I_{n-k_0} & 0 &-\tilde{X}_1\tilde{Y}_1^{-1} \\ 0 & 0& I_{k_0} & 0 \\ 0&0&0& I_{n-k_0}\end{array}\right)$. Obviously, $M_3$ is a symplectic orthogonal matrix.

Let $\bar{V}_0=span\{e_1,\cdots,e_n\}$,  $\bar{V}_1=span\{e_{k_0+1},\cdots,e_{k_0+n}\}$ which are Lagrangian subspaces. Let $P_0$, $P_1$ be   the orthogonal  projections onto $\bar{V}_0$ and  $\bar{V}_1$ respectively.  For any matrix  $M$  on $\mathbb R^{2n}$, we always set
$$\cal{P}(M):=P_1M_3MM_3^{-1}P_0, $$ which is a $n\times n$ matrix with $ \cal{P}(M)_{i,j}=(M_3MM_3^{-1} e_j, e_{i+k_0})$.
In the case $\dim V_0=0$ or $n$,  the expression of $\cal{P}(M)$ is simple. In fact,
rewrite
$M_3MM_3^{-1}=\left(\begin{array}{cc} \hat{M}_1& \hat{M}_2\\
\hat{M}_3  & \hat{M}_4 \end{array}\right) $, then  $\cal{P}(M)=\hat{M}_3$ in the case $\Lambda_0=\Lambda_1$,  and $\cal{P}(M)=\hat{M}_1$ in the  transversal case $\Lambda_0\cap \Lambda_1=\{0\}$.

Notice that $\det (M_3)=1$,  $$g(\lambda)=\det(M_3)\det(\ga_\lambda(T)Z_0,Z_1)= \det(M_3\ga_\lambda(T) Z_0, M_3Z_1).$$ Direct computation shows that
$$\det(M_3\ga_\lambda(T) Z_0, M_3Z_1)=\det(M_3\ga_\lambda(T)M_3^{-1}M_3Z_0,M_3Z_1)=(-1)^{nk_0}\det(\cal{P}(\ga_\lambda(T)))\det(\tilde{Y}_1). $$
Then \bea g(\lambda)g(0)^{-1}=\det(\cal{P}(\ga_\lambda(T))\cdot \det(\cal{P}(\ga_0(T))^{-1}).  \eea

Set $D=diag(0_n, -R_1 )$, then $B_\lambda=B_0+\lambda D$, from \cite{HOW}, let  \begin{eqnarray*}\hat{D}(t)=\gamma_{0}^T(t)D(t) \gamma_{0}(t),\end{eqnarray*} and
\begin{eqnarray*}
F_k=\int_0^TJ\hat{D}(t_1)\int_0^{t_1}J\hat{D}(t_2)\cdots\int_0^{t_{k-1}}J\hat{D}(t_k)dt_k\cdots dt_2dt_1,
k\in\mathbb N\label{adc4.13}.  \end{eqnarray*}
 By  Taylor's formula,
\begin{eqnarray*}
\gamma_{\lambda}(T)=\gamma_0(T)(I_{2n}+\lambda F_1+\cdots+\lambda^kF_k
+\cdots), \label{c4.11}
\end{eqnarray*}
then
\begin{eqnarray*}
\cal{P}(\gamma_{\lambda}(T))=\cal{P}(\gamma_{0}(T))+\lambda \cal{P}(\gamma_{0}(T) F_1)+\cdots+\lambda^k \cal{P}(\gamma_{0}(T)F_k)
+\cdots), \label{c4.11}
\end{eqnarray*}
where $\cal{P}(\gamma_{0}(T))$ is nonsingular.  Set \bea
G_k=\cal{P}(\gamma_{0}(T) F_k)\cdot \cal{P}(\gamma_{0}(T))^{-1}, \,\
for \,\ k\in \mathbb N,\lb{G}\eea and let $
f(\lambda)=\det(I_n+\lambda G_1+\cdots)$,  which is an analytic
function of $\lambda$.  It is obvious that $
g(\lambda)g(0)^{-1}=f(\lambda)$.

Since $f(\lambda)$ vanishes nowhere near $0$, we can write  $f(\lambda)=e^{g(\lambda)}$, then by  \cite[Formula (2.6)]{HOW} and some direct computation,
\begin{eqnarray}  g^{(m)}(0)/m != \sum_{k=1}^m \frac{(-1)^{k+1}}{k}\Big(\sum_{j_1+\cdots+j_k=m}Tr(G_{j_1}\cdots
G_{j_k})\Big). \label{cc4.22}  \end{eqnarray} Compare the
coefficients in  (\ref{cc4.20}) with  (\ref{cc4.22}), we have \bea
Tr((R_1A^{-1})^m)= m\sum_{k=1}^m
\frac{(-1)^{k+m}}{k}\Big(\sum_{j_1+\cdots+j_k=m}Tr(G_{j_1}\cdots
G_{j_k})\Big).   \eea This proves Theorem \ref{thm1.2}  because $Tr((R_1A^{-1})^m)=\sum\limits_{j}{(-1)^m\over \lambda_j^m}    $ .

Moreover, for the first two terms, we can write it more precisely.
\begin{eqnarray}\label{eq3.49ab}
\sum\limits_{j}{1\over \lambda_j}=-Tr(G_1)=-Tr\Big(\cal{P}\(\ga_0(T)\cdot
J\int_0^T\gamma_{0}^T(t)D(t) \gamma_{0}(t)dt\)\cdot
\cal{P}(\ga_0(T))^{-1}\Big), \lb{t.1}
\end{eqnarray}
and
\begin{eqnarray}
\sum\limits_{j}{1\over \lambda_j^2}&=&Tr(G_1^2)-2Tr(G_2) \nonumber
\\&=&
-2Tr \Big(\cal{P}\(\ga_0(T)\cdot J\int_0^T\gamma_{0}^T(t)D(t) \gamma_{0}(t)J\int_0^s\gamma_{0}^T(s)D(s) \gamma_{0}(s)ds dt\) \cdot \cal{P}(\ga_0(T))^{-1} \Big)\nonumber\\
                                &&\ \ \ +Tr\Big(\Big[\cal{P}\(\ga_0(T)\cdot J\int_0^T\gamma_{0}^T(t)D(t)
\gamma_{0}(t)dt\)\cdot \cal{P}(\ga_0(T))^{-1}\Big]^2\Big). \label{eq2.14}
\end{eqnarray}

\section{Examples}
In this section, we will give some detailed calculation on the trace formula for some special separated boundary value problems for Sturm-Liouville system. At first, we will consider the Dirichlet problem for the system (\ref{4.0}).   Obviously, in this case   $A=-{d^2\over dt^2}$, $R_1=-R$.
  Let  $K_n=\left(\begin{array}{cc} I_n & 0_n \\ 0_n & 0_n\end{array}\right)$, $D=\left(\begin{array}{cc} 0_n & 0_n \\ 0_n & R\end{array}\right)$.
Recall that
  $\gamma_{0}(t)$ satisfied
 $\dot{\gamma}_0(t)=JK_n\gamma_{0}(t)$ with $\gamma_{0}(0)=I_{2n}$. Direct computation shows that
 $ \gamma_0(t)=\left(\begin{array}{cc}I_n & 0_n\\ tI_n&I_n \end{array}\right).$ It is easy to verify  $\gamma_{0}(t)^{-1}=\left(\begin{array}{cc}I_n & 0_n\\ -tI_n&I_n \end{array}\right)$. Therefore,
 $$ J\hat{D}(t)= \gamma_{0}^{-1}(t)JD(t)\gamma_{0}(t)=\left(\begin{array}{cc} -tR(t) & -R(t)\\ t^2R(t)&tR(t) \end{array}\right).    $$
Then
$$  J\int_0^T \hat{D}dt= \left(\begin{array}{cc}  -\int_0^TtRdt &  -\int_0^TRdt\\  \int_0^Tt^2Rdt&  \int_0^TtRdt \end{array}\right),$$
and \bea \ga_0(T)\cdot J\int_0^T \hat{D}dt= \left(\begin{array}{cc}  -\int_0^TtRdt &  -\int_0^TRdt\\  \int_0^Tt^2Rdt-T\int_0^TtRdt&  \int_0^TtRdt-T\int_0^TRdt \end{array}\right). \lb{4.1}   \nonumber \eea
Obviously,  $\cal{P}(\ga_0(T))=TI_n$, and  \bea \cal{P}\( \ga_0(T)J\int_0^T \hat{D}dt \)=\int_0^Tt^2Rdt-T\int_0^TtRdt,  \nonumber\eea
thus \bea G_1=\frac{1}{T}\int_0^Tt^2Rdt-\int_0^TtRdt.  \nonumber\eea
We have \bea Tr(RA^{-1})=\sum_j\frac{1}{\lambda_j}=Tr\Big(\int_0^T\(t-\frac{t^2}{T}\)Rdt\Big). \nonumber\eea

 Recall that $ R^+=\frac{1}{2}(R+|R|)$ is nonnegative matrices. let $\lambda^+_j$ be the $j$-th eigenvalue of $\ddot{y}+\lambda R^+y=0$ under the Dirichlet boundary conditions,   then $\lambda_j>0$ for $j\in\mathbb{N}$.  Similar to the discussion of \cite[Theorem 4.12]{HOW},
 $\sum_j\frac{1}{\lambda_j}=Tr\Big(\int_0^T\(t-\frac{t^2}{T}\)R^+dt\Big)<1$ implies $\lambda_1>1$, hence
$\sd+R^+$ is nondegenerate for $\lambda\in[0,1]$.  Since $R^+\geq R$,  we have    $\sd+R$ is
nondegenerate. This proves Corollary \ref{cr1.3}.

At the end of this paper, we will consider  (\ref{4.0}) with the boundary condition (\ref{mix}).
We choose $d_1=\left(\begin{array}{cc}1
\\ 0 \end{array}\right)$ and  $f_2=\left(\begin{array}{cc}\cos(\theta)
\\ -\sin(\theta) \end{array}\right)$ to be the frame of $\Lambda_0$ and $\Lambda_1$ separately.  Then $M_1=I_2$,  $M_2=\left(\begin{array}{cc} 1 & \cot(\theta) \\ 0 & 1\end{array}\right)$, and consequently $M_3=M_2$.
It is not hard to see,
$ M_3^{-1}=\left(\begin{array}{cc} 1 & -\cot(\theta) \\ 0 & 1\end{array}\right)   $. Rewrite  $M=\left(\begin{array}{cc} a & b \\ c & d\end{array}\right)$ in short. Direct computation shows that in this case
\bea \cal{P}(M)=a+c
\cot(\theta).  \lb{pm}\eea
So we have $\cal{P}(\ga_0(T))=1+T\cot(\theta)$, and easy computations show that
\bea \cal{P}\( \ga_0(T)J\int_0^T \hat{D}dt \)=-\frac{T^2}{2}-\frac{T^3}{6} \cot(\theta).  \eea
We get \bea G_1=-\frac{3T^2+T^3\cot(\theta)}{6(1+T\cot(\theta))}=-\frac{3T^2\sin(\theta)+T^3\cos(\theta)}{6(\sin(\theta)+T\cos(\theta))}. \eea
By (\ref{t.1}), we get (\ref{e.e}). It should be pointed out that maybe the identity (\ref{e.e}) could  obtained by some other method. However, by the trace formula, we can get many other interesting identities directly  if we consider  different boundary conditions for the Sturm-Liouville system. 

\medskip

\noindent {\bf Acknowledgements.} The  authors thank Y. Long sincerely for his encouragements and interests.

\end{document}